\documentclass[12pt]{amsart}

\usepackage{comment}
\usepackage{amsthm,amssymb,amsmath,amstext,amsfonts}
\usepackage{enumitem,mathtools,pgfplots,pgfmath}
\usepackage{graphicx}
\usepackage{subcaption}
\usepackage[export]{adjustbox}
\usepackage{tikz}
\usepackage{amsmath}
\usepackage{pseudo}
\usetikzlibrary{arrows,shapes,automata,backgrounds,decorations,petri,positioning}
\usetikzlibrary{decorations.pathreplacing,angles,quotes}
\tikzset{every loop/.style={min distance=8.8mm,looseness=4}}

\pgfdeclarelayer{background}
\pgfsetlayers{background,main}
\pgfdeclarelayer{background}
\pgfsetlayers{background,main}

\usepackage[margin=1in,letterpaper,portrait]{geometry}
\usepackage[latin1]{inputenc}
\usepackage{url}
\usepackage{amsaddr}

\theoremstyle{plain}
\theoremstyle{definition}
\newtheorem{theorem}{Theorem}

\newtheorem{lemma}{Lemma}
\newtheorem{definition}{Definition}

\newtheorem{corollary}{Corollary}

\DeclareMathAlphabet{\mathpzc}{OT1}{pzc}{m}{it}

\pgfplotsset{compat=1.16}
\begin{document}

\title{Digraphs with exactly one Eulerian tour}

\author{Luz Grisales$^1$, 
Antoine Labelle$^2$, Rodrigo Posada$^1$, Stoyan Dimitrov$^3$}
\address{$^1$ Massachusetts Institute of Technology}
\address{$^2$ Coll\`ege de Maisonneuve}
\address{$^3$ University of Illinois at Chicago}
\email{luzg@mit.edu, antoinelab01@gmail.com, rposada@mit.edu, sdimit6@uic.edu}

\begin{abstract}
We give two combinatorial proofs of the fact that the number of loopless directed graphs (digraphs) on the vertex set $[n]$ with no isolated vertices and with exactly one Eulerian tour up to a cyclic shift is $\frac{1}{2}(n-1)!C_{n}$, where $C_{n}$ denotes the $n$-th Catalan number. We construct a bijection with a set of labeled rooted plane trees and with a set of valid parenthesis arrangements. 
\end{abstract}

\maketitle

\section{Introduction and main facts}
Richard Stanley has a list containing nearly 250 problems and facts, for which he asks of combinatorial proofs \cite{problems}. This work describes two such proofs for one of the problems in the list, namely Problem 199, without a known combinatorial proof. First, we recall some definitions following \cite{algComb}.

A (finite) directed graph (or digraph) $D$ consists of a vertex set $V = \{v_{1},\ldots ,v_{n}\}$ and an edge set $E = \{e_{1},\ldots ,e_{q}\}$, together with a function $\phi : V\to V$ determining the direction of each edge. If $\phi(e) = (u,v)$, then we think of $e$ as an arrow from $u$ to $v$. We will call $u$ - initial vertex and $v$ - final vertex. The outdegree of a vertex $v$, denoted $outdeg(v)$, is the number of edges of $D$ with initial vertex $v$. Similarly, the indegree of $v$, denoted $indeg(v)$, is the number of edges of $D$ with final vertex $v$. A loop is an edge $e$ for which $\phi(e) = (v,v)$ for some vertex $v$. A digraph is balanced if $indeg(v) = outdeg(v)$ for each of its vertices $v$. An oriented path in a digraph $D$ is a sequence of vertices $v_{1},\ldots , v_{m}$, where $(v_{i},v_{i+1})$ is an edge of $D$ for each $i\in [m-1]$. If the vertices $v_{1},\ldots , v_{m}$ are all different, then we call the path \emph{simple}. If we have a simple path and $v_{m} = v_{1}$, then we have an \emph{oriented simple cycle}.

\begin{definition}
An \textit{Eulerian tour} in a directed graph $D$ is a sequence of vertices $a_1a_2\cdots a_k$ such that $(a_1, a_2), \ (a_2, a_3), \cdots, (a_{k-1}, a_k), \ (a_k, a_1)$ are all the distinct directed edges of $D$. 
\end{definition}

Any cyclic shift $a_ia_{i+1}\cdots a_ka_1\cdots a_{i-1}$ of an Eulerian tour is also an Eulerian tour and we will say that these tours are \textit{equivalent up to a cyclic shift}. 

\begin{definition}
An \textit{Eulerian digraph} is a digraph which has no isolated vertices and contains exactly one Eulerian tour (and its equivalents under cyclic shift).
\end{definition}

Our goal is to prove the following claim.

\begin{theorem}
\label{th:main}
If $A_{n}$ is the set of loopless \textit{Eulerian digraphs}  on the vertex set $[n]$, then \\ $|A_{n}| = \frac{1}{2}(n-1)!C_{n}$ (sequence A102693 in OEIS \cite{OEIS}), where $C_{n} = \frac{1}{n+1}\binom{2n}{n}$ denotes the $n$-th Catalan number.
\end{theorem}

For example, $|A_{3}| = 5$. Indeed, there are two such digraphs that look like triangles and three that consist of two 2-cycles with a common vertex (see Figure \ref{fig:A3}). The related OEIS sequence, A102693, was created by Richard Stanley. As a reference, he points out to an unpublished work of him. Thus, we can assume that Theorem \ref{th:main} was first proved there.

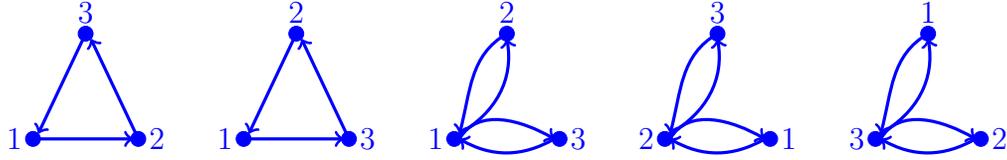
\begin{figure}
\begin{tikzpicture}[scale = 1.4]

  \filldraw[blue] (-5,0) circle (2pt) node[anchor=east] {1};
  \draw[blue, very thick, ->]  (-5,0) -- (-4.05,0);
  
  \filldraw[blue] (-4,0) circle (2pt) node[anchor=west] {2};
  \draw[blue, very thick, ->]  (-4,0) -- (-4.45,0.95);
  
  \filldraw[blue] (-4.5,1) circle (2pt) node[anchor=south] {3};
  \draw[blue, very thick, ->]  (-4.5,1) -- (-4.95,0.05);

  \filldraw[blue] (-3,0) circle (2pt) node[anchor=east] {1};
  \draw[blue, very thick, ->]  (-3,0) -- (-2.05,0);
  
  \filldraw[blue] (-2,0) circle (2pt) node[anchor=west] {3};
  \draw[blue, very thick, ->]  (-2,0) -- (-2.45,0.95);
  
  \filldraw[blue] (-2.5,1) circle (2pt) node[anchor=south] {2};
  \draw[blue, very thick, ->]  (-2.5,1) -- (-2.95,0.05);

  \filldraw[blue] (-1,0) circle (2pt) node[anchor=east] {1};
  \draw[blue, very thick, ->]  (-0.95,0) to [out=60,in=150] (-0.05,0);
  \draw[blue, very thick, ->]  (-0.05,0) to [out=210,in=-30] (-0.95,0);
  
  \filldraw[blue] (0,0) circle (2pt) node[anchor=west] {3};
  
  \filldraw[blue] (-0.5,1) circle (2pt) node[anchor=south] {2};
  \draw[blue, very thick, ->]  (-0.5,1) to [out=210,in=75] (-0.95,0.05);
  \draw[blue, very thick, ->]  (-0.95,0.05) to [out=30,in=-80] (-0.5,0.95);

  \filldraw[blue] (1,0) circle (2pt) node[anchor=east] {2};
  \draw[blue, very thick, ->]  (1.05,0) to [out=60,in=150] (1.95,0);
  \draw[blue, very thick, ->]  (1.95,0) to [out=210,in=-30] (1.05,0);
  
  \filldraw[blue] (2,0) circle (2pt) node[anchor=west] {1};
  
  \filldraw[blue] (1.5,1) circle (2pt) node[anchor=south] {3};
  \draw[blue, very thick, ->]  (1.5,1) to [out=210,in=75] (1.05,0.05);
  \draw[blue, very thick, ->]  (1.05,0.05) to [out=30,in=-80] (1.5,0.95);

  \filldraw[blue] (3,0) circle (2pt) node[anchor=east] {3};
  \draw[blue, very thick, ->]  (3.05,0) to [out=60,in=150] (3.95,0);
  \draw[blue, very thick, ->]  (3.95,0) to [out=210,in=-30] (3.05,0);
  
  \filldraw[blue] (4,0) circle (2pt) node[anchor=west] {2};
  
  \filldraw[blue] (3.5,1) circle (2pt) node[anchor=south] {1};
  \draw[blue, very thick, ->]  (3.5,1) to [out=210,in=75] (3.05,0.05);
  \draw[blue, very thick, ->]  (3.05,0.05) to [out=30,in=-80] (3.5,0.95);
\end{tikzpicture}   
\caption{The five digraphs with three vertices and a unique Eulerian tour.}
\label{fig:A3}
\end{figure}

It is not difficult to show that every Eulerian graph must be connected and balanced \cite[Theorem 10.1]{algComb}. The BEST theorem that we recall below gives us a formula for the total number of Eulerian tours in a digraph. In order to understand this result, one should be familiar with the term \emph{oriented tree} (see Figure \ref{fig:exOrT}). An oriented tree with root $v$ is a finite digraph $T$ with $v$ as one of its vertices, such that there is a unique directed path from any other vertex of $T$ to $v$. This means that the underlying undirected graph (after we erase all the arrows of the edges of $T$) is a tree.

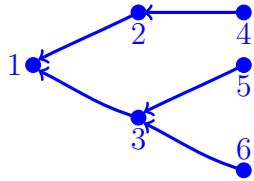
\begin{figure}[h!]
\begin{tikzpicture}[scale = 1.4]
  \filldraw[blue] (0,0) circle (2pt) node[anchor=east] {1};
  \filldraw[blue] (1,0.5) circle (2pt) node[anchor=north] {2};
  \draw[blue, very thick, ->]  (1,0.5) to [out=206.5, in = 26.5] (0.05,0.05);
  \filldraw[blue] (2,0.5) circle (2pt) node[anchor=north] {4};
  \draw[blue, very thick, ->]  (2,0.5) to [out=180, in = 0] (1.05,0.5);
  
  \filldraw[blue] (1,-0.5) circle (2pt) node[anchor=north] {3};
  \draw[blue, very thick, ->]  (1,-0.5) to [out=163.5, in = -26.5] (0.05,-0.05);
  
  \filldraw[blue] (2,0) circle (2pt) node[anchor=north] {5};
  \draw[blue, very thick, ->]  (2,0) to [out=206.5, in = 26.5] (1.05,-0.45);
  
  \filldraw[blue] (2,-1) circle (2pt) node[anchor=south] {6};
  \draw[blue, very thick, ->]  (2,-1) to [out=163.5, in = -26.5] (1.05,-0.55);
  
\end{tikzpicture}   
\caption{Example of an oriented tree.}
\label{fig:exOrT}
\end{figure}

\begin{theorem}[BEST theorem, \cite{algComb}]
\label{th:BEST}
Let $D$ be a connected balanced digraph with vertex set $V$. Fix an edge $e$ in $V$ and let $v$ be the initial vertex of that edge. Let $\tau (D, v)$ denote the number of oriented (spanning) subtrees of $D$ with root $v$, and let $\epsilon (D, e)$ denote the number of Eulerian tours of $D$ starting with the edge $e$. Then
$$\epsilon (D, e) = \tau (D, v) \prod_{u \in V} (outdeg(u) -1)!$$
\end{theorem} 

\begin{corollary}[from Theorem \ref{th:BEST}]
\label{corr:BEST}
A digraph $D \in A_n$ if and only if
\begin{enumerate}
    \item For every vertex $v$, $D$ has exactly one oriented (spanning) subtree with root $v$.
    \item The outdegree of an arbitrary vertex of $D$ is 1 or 2.\\
\end{enumerate}
\end{corollary}

Using Corollary \ref{corr:BEST}, we will characterize the digraphs in $A_{n}$ by two other conditions that will be used later.
\begin{lemma}
\label{lemma:1}
A digraph $D \in A_n$ if and only if
\begin{enumerate}
    \item[i)] There exists a unique oriented simple path between any two vertices of $D$.
    \item[ii)] Every vertex of $D$ is part of exactly one or two simple oriented cycles.
    
\end{enumerate}
\end{lemma}
\begin{proof}
$[$First part: $D \in A_{n}$ $\Longrightarrow$ conditions $i)$ and $ii)]$ Let $D \in A_{n}$. Let $u$ and $v$ be two arbitrary vertices in $D$. By condition $(1)$ of Corollary \ref{corr:BEST}, there exists exactly one oriented spanning subtree $T_{u}$ of $D$ with root $u$. We know that $v$ has to be a vertex of $T_{u}$ and that there exists a unique simple path $\mathcal{P}$ from $v$ to $u$ in $T_{u}$, which is a simple oriented path in $D$. Assume that there exists another path $\mathcal{P'}\neq \mathcal{P}$ between $v$ and $u$ in $D$. Begin from $v$ and follow $\mathcal{P}'$. Let $f'$ be the first edge in $\mathcal{P}'$ which is not part of $\mathcal{P}$ and let $f$ be the edge in $\mathcal{P}$ with the same initial vertex as $f'$. If you delete $f$ from $T_{u}$ and add $f'$ to it, you will obtain a graph $T'_{u}$. One can easily see that $T'_{u}$ is an oriented spanning subtree of $D$, different from $T_{u}$ (see Figure \ref{fig:lemma}). This is a contradiction. Thus, we showed that condition $i)$ holds. 

Corollary \ref{corr:BEST} implies that each vertex of $D$ can be part of at most two simple oriented cycles since its outdegree is $1$ or $2$. It remains to show that each vertex of $D$ is part of at least one such cycle. Let $v$ be an arbitrary vertex of $D$ and let $(u,v)$ be an edge of $D$ (such an edge exists since the indegree of $v$ is $1$ or $2$). We showed that there is a unique oriented simple path between $v$ and $u$. This path together with the edge $(u,v)$ forms a simple oriented cycle. Thus condition $ii)$ holds. 

\begin{figure}[h!]
\begin{tikzpicture}[scale = 0.85]
  \filldraw[blue] (0,4) circle (2pt) node[anchor=south] {1};
  \filldraw[blue] (-2,2) circle (2pt) node[anchor=south] {2};
  \filldraw[blue] (-3,1) circle (2pt) node[anchor=south] {3};
  \filldraw[blue] (-1,1) circle (2pt) node[anchor=south] {4};
  
  \filldraw[blue] (2,2) circle (2pt) node[anchor=south] {5};
  \filldraw[blue] (1,1) circle (2pt) node[anchor=north] {6};
  \filldraw[blue] (2,1) circle (2pt) node[anchor=east] {7};
  \filldraw[blue] (3,1) circle (2pt) node[anchor=east] {8};
  
  \filldraw[blue] (0,0) circle (2pt) node[anchor=south] {9};
  \filldraw[blue] (-1,-1) circle (2pt) node[anchor=north] {11};
  \filldraw[blue] (0,-1) circle (2pt) node[anchor=north] {12};
  \filldraw[blue] (1,-1) circle (2pt) node[anchor=north] {13};
  
  \filldraw[blue] (2,0) circle (2pt) node[anchor=west] {10};
  
  \draw[blue, very thick]  (0,4) to [out=225, in = 45] (-2,2);
  \draw[blue, very thick, ->]  (2,2) to [out=135, in = -45] (0.05,3.95);
  \draw[blue, very thick]  (-2,2) to [out=45, in = 225] (-3,1);
  \draw[blue, very thick]  (-2,2) to [out=-45, in = 135] (-1,1);
  
  \draw[blue, very thick, ->]  (1,1) to [out= 45, in = 225] (1.95,1.95);
  \draw[blue, very thick]  (2,2) to [out=-90, in = 90] (2,1);
  \draw[blue, very thick]  (2,2) to [out=-45, in = 135] (3,1);
  
  \draw[blue, very thick, ->]  (0,0) to [out=45, in = 225] (0.95,0.95);
  \draw[blue, very thick]  (1,1) to [out=-45, in = 135] (2,0);
  
  \draw[blue, very thick, ->]  (-1,-1) to [out=45, in = 225] (-0.05,-0.05);
  \draw[blue, very thick]  (0,0) to [out=-90, in = 90] (0,-1);
  \draw[blue, very thick]  (0,0) to [out=-45, in = 135] (1,-1);
  
  \draw[blue, dashed, very thick, ->]  (1,1) to [out=161.5, in = -18.5] (-1.88,1.96);
  
  \filldraw[blue] (0.5,1.7) circle (0pt) node[anchor=east] {$f'$};
  \filldraw[blue] (1.35,1.25) circle (0pt) node[anchor=south] {$f$};
  
\end{tikzpicture}   
\caption{The tree $T_{u}$ in the first part of the proof of Lemma \ref{lemma:1}; $u = 1$, $v=11$, $\mathcal{P} = 11,9,6,5,1$, $\mathcal{P'} = 11,9,6,2,1$, $f' = (6,2)$ and $f = (6,5)$. Delete $f$ and add $f'$ to obtain another tree $T'_{u}$.}
\label{fig:lemma}
\end{figure}
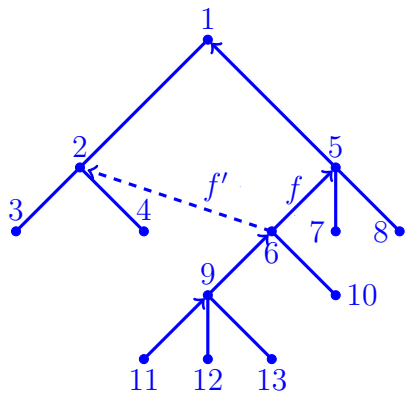

$[$second part: conditions $i)$ and $ii)$ $\Longrightarrow$ $D \in A_{n}]$ Take a digraph $D$ for which conditions $i)$ and $ii)$ hold. We have to show that conditions $(1)$ and $(2)$ from Corollary \ref{corr:BEST} also hold. Let $v$ be an arbitrary vertex of $D$. Condition $ii)$ implies that $outdeg(v)\geq 1$. We will show that $outdeg(v)< 3$. Condition $i)$ implies that no edge of $D$ can be part of two different simple cycles. Indeed, assume that $(u,w)$ is an edge of $D$, which is a part of two different simple oriented cycles. Then, we must have at least two different simple oriented paths between $w$ and $u$, which contradicts condition $i)$. Now, assume that $outdeg(v)\geq 3$ and let $(v,u_{1})$, $(v,u_{2})$ and $(v,u_{3})$ are three different edges of $D$. We know that there exist simple paths between $u_{i}$ and $v$, for $i=1,2,3$. Thus, $v$ participates in simple cycles through $u_{i}$, for $i=1,2,3$ and no two of these simple cycles share and edge. Therefore these three cycles are different. This is a contradiction with condition $ii)$.

It remains to show that condition $(1)$ from Corollary \ref{corr:BEST} holds. Take an arbitrary vertex $v$ of $D$. We have a unique oriented simple path from each of the other vertices of $D$ to $v$. Take the union of these paths. The graph that you will obtain is an oriented spanning subtree $T_{v}$ of $D$ with root $v$. Assume that there exists another such subtree $T'_{v}$. Then, since $T_{v}\neq T'_{v}$, we must have a vertex $w$ in $D$, for which the unique oriented path from $w$ to $v$ in $T'_{v}$ is different from the unique oriented path from $w$ to $v$ in $T_{v}$. These are two different oriented paths from $w$ to $v$ in $D$, which is a contradiction.
\end{proof}
\begin{corollary}
\label{corr:lem1}
If $D\in A_{n}$, then no pair of cycles in $D$ have an edge in common.
\end{corollary}
\begin{proof}
Assume $C_{1}$ and $C_{2}$ are two different cycles in $D$, which have the edge $(u,v)$ in common. Then, we will have at least two oriented paths from $v$ to $u$ in $D$ - one following $C_{1}$ and another one following $C_{2}$. This is a contradiction with Corollary \ref{corr:BEST}.
\end{proof}

\section{Bijection with a set of labeled rooted plane trees}
We will construct a bijection between the digraphs in $A_{n}$ and a set of labeled rooted plane trees on $n+1$ vertices. Let $U_n$ be the set of the unlabeled rooted plane trees with $n+1$ vertices. It is well-known that $|U_n|=C_n$ (see \cite[Theorem 1.5.1]{catalan}). Let $L_n$ be the set of the labeled rooted plane trees with $n+1$ vertices such that the root is always labeled as $0$ and the left-most child of the root is always labeled as $1$. We have $|L_n|= (n-1)!|U_n| = (n-1)!C_n$. Finally, let $L'_{n}\subset L_{n}$ be the set of the labeled trees in $L_{n}$, such that the vertex $2$ is in the subtree with root $1$. 

First, we define a map $f$ over $L_{n}$, such that $f(L'_{n})=L_{n}\setminus L'_{n}$ and $f(f(T)) = T$, i.e., an involution. This shows that $|L'_{n}| = |A_{n}|$. Then, we define a map $g: L'_{n}\to A_{n}$, which is shown to be a bijection. Below, we will denote by $T_{(x,j)}$ the $j$-th child (from left to right) of the vertex $x$ of a tree $T \in L_n$. 

\begin{definition}
Let $f:L_n \rightarrow L_n$ be a map, which switches the places of the subtree with root $1$ (excluding $1$) and the subtree with root $0$ (excluding $0$ and the subtree with root $1$), for every tree in $L_n$ (see Figure \ref{fig:f}). Formally, if $T \in L_n$, then $f(T)$ has the following properties. \\
For every $j\geq 1$:
\begin{itemize}
    \item If the vertex $v$ is the (j+1)-th child of the vertex $0$ in $T$, then $v$ is the j-th child of vertex $1$ in $f(T)$, i.e., $T_{(0, j+1)}=f(T)_{(1, j)}$.
    \item If the vertex $v$ is the j-th child of the vertex $1$ in $T$, then $v$ is the (j+1)-th child of vertex $0$ in $f(T)$, i.e., $T_{(1, j)}=f(T)_{(0, j+1)}$.
    \item All the other directed edges are left the same for both trees, i.e., $T_{(u, j)}=f(T)_{(u, j)}$ for $u \not \in \{0,1\}$. 
\end{itemize}
\end{definition}

\begin{figure}[h!]
\centering
\begin{tikzpicture}[scale = 0.85]
    \draw[black] (2,-3) node {$T$};
  \draw[red, very thick]  (0,0) -- (1,-1); 
  \filldraw[red] (0,0) circle (2pt) node[anchor=east] {2};
  \draw[red, very thick]  (2,0) -- (1,-1); 
  \filldraw[red] (2,0) circle (2pt) node[anchor=west] {3};
  \draw[blue, very thick]  (3,0) -- (3,-1); 
  \filldraw[blue] (3,0) circle (2pt) node[anchor=west] {6};
  \draw[blue, very thick]  (3,-1) -- (2,-2); 
  \filldraw[blue] (3,-1) circle (2pt) node[anchor=west] {5};
  \draw[blue, very thick]  (2,-1) -- (2,-2); 
  \filldraw[blue] (2,-1) circle (2pt) node[anchor=west] {4};
  \draw[black, very thick]  (1,-1) -- (2,-2); 
  \filldraw[black] (1,-1) circle (2pt) node[anchor=east] {1};
  \filldraw[black] (2,-2) circle (2pt) node[anchor=north] {$0$};
  \draw[->, ultra thick] (4,-1) -- (5,-1);
  
  \draw[black] (8,-3) node {$f(T)$};
  \draw[blue, very thick]  (8,1) -- (8,0); 
  \filldraw[blue] (8,1) circle (2pt) node[anchor=west] {6};
  
  \draw[blue, very thick]  (6,0) -- (7,-1); 
  \filldraw[blue] (6,0) circle (2pt) node[anchor=east] {4};
  \draw[blue, very thick]  (8,0) -- (7,-1); 
  \filldraw[blue] (8,0) circle (2pt) node[anchor=west] {5};
  \draw[red, very thick]  (9,-1) -- (8,-2); 
  \filldraw[red] (9,-1) circle (2pt) node[anchor=west] {3};
  \draw[red, very thick]  (8,-1) -- (8,-2); 
  \filldraw[red] (8,-1) circle (2pt) node[anchor=west] {2};
  \draw[black, very thick]  (7,-1) -- (8,-2); 
  \filldraw[black] (7,-1) circle (2pt) node[anchor=east] {1};
  \filldraw[black] (8,-2) circle (2pt) node[anchor=north] {$0$};
\end{tikzpicture}
\caption{Example of the action of the map $f$.} 
\label{fig:f}
\end{figure}
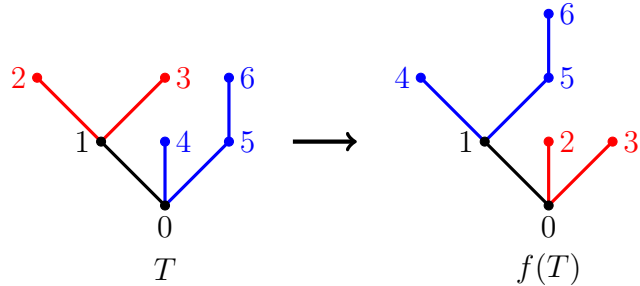

Note that $f(L'_{n})=L_{n}\setminus L'_{n}$ and $f(f(T)) = T$, i.e., $f^{-1}=f$. Therefore, $|L'_{n}| = \frac{|L_{n}|}{2}= \frac{(n-1)!C_{n}}{2} = |A_{n}|$. 

\begin{definition}
Let $g: L'_{n}\to A_{n}$ be a map, such that if $T \in L'_n$, $g(T) = D'$ is a digraph with $V(D') = [n]$ and $E(D')$, such that if $x$ is a vertex of $T$ with $r$ children, then:
\begin{enumerate}
    \item For every $i\in [1,r), \ (T_{(x,i)}, T_{(x,i+1)}) \in E(D')$.
    \item If $x=0$, then $(T_{(x,r)},T_{(x,1)}) \in E(D')$.
    \item If $x \not= 0$, then $(T_{(x,r)}, x) \in E(D')$ and $(x,T_{(x,1)}) \in E(D')$.\\
\end{enumerate}
\end{definition}

\begin{figure}[h!]
    \centering
\begin{tikzpicture}
        \draw[blue] (2,-3) node {$T$};
  \draw[blue, very thick]  (0,0) -- (1,-1); 
  \filldraw[blue] (0,0) circle (2pt) node[anchor=east] {2};
  \draw[blue, very thick]  (2,0) -- (1,-1); 
  \filldraw[blue] (2,0) circle (2pt) node[anchor=west] {3};
  \draw[blue, very thick]  (1,0) -- (1,-1); 
  \filldraw[blue] (1,0) circle (2pt) node[anchor=west] {7};
  \draw[blue, very thick]  (1,0) -- (0.4,1); 
  \filldraw[blue] (0.4,1) circle (2pt) node[anchor=east] {9};
  \draw[blue, very thick]  (1,0) -- (1.6,1); 
  \filldraw[blue] (1.6,1) circle (2pt) node[anchor=west] {8};
  \draw[blue, very thick]  (3,0) -- (3,-1); 
  \filldraw[blue] (3,0) circle (2pt) node[anchor=west] {6};
  \draw[blue, very thick]  (3,-1) -- (2,-2); 
  \filldraw[blue] (3,-1) circle (2pt) node[anchor=west] {5};
  \draw[blue, very thick]  (2,-1) -- (2,-2); 
  \filldraw[blue] (2,-1) circle (2pt) node[anchor=west] {4};
  \draw[blue, very thick]  (1,-1) -- (2,-2); 
  \filldraw[blue] (1,-1) circle (2pt) node[anchor=east] {1};
  \filldraw[blue] (2,-2) circle (2pt) node[anchor=north] {$0$};
  \draw[->, ultra thick] (4.75,-1) -- (5.75,-1);
  \draw[blue] (9,-3) node {$g(T)$};
  \filldraw[blue] (7,0) circle (2pt) node[anchor=east] {2};
  \filldraw[blue] (7.80,0) circle (2pt) node[anchor=north] {7};
  \filldraw[blue] (7.40,1) circle (2pt) node[anchor=east] {9};
  \filldraw[blue] (8.20,1) circle (2pt) node[anchor=west] {8};
  \filldraw[blue] (8.60,0) circle (2pt) node[anchor=west] {3};
  \filldraw[blue] (10,0) circle (2pt) node[anchor=south] {6};
  \filldraw[blue] (10,-1.2) circle (2pt) node[anchor=north west] {5};
  \filldraw[blue] (9,-1.2) circle (2pt) node[anchor=south] {4};
  \filldraw[blue] (7.9,-1.2) circle (2pt) node[anchor=east] {1};
  \filldraw[blue] (9,-2.4) circle (2pt);
  \draw[blue, very thick, ->]  (7.8,-1.06) -- (7.1,-0.12);
  \draw[blue, very thick, ->]  (7.08,0) -- (7.72,0);
  \draw[blue, very thick, ->]  (7.88,0) -- (8.52,0);
  \draw[blue, very thick, ->]  (7.76,0.1) -- (7.44,0.9);
  \draw[blue, very thick, ->]  (8.16,0.9) -- (7.84,0.1);
  \draw[blue, very thick, ->]  (7.5,1) -- (8.15,1);
  \draw[blue, very thick, ->]  (8.5,-0.13) -- (7.96,-1.05);
  \draw[blue, very thick, ->]  (7.9,-1.2) -- (8.8,-1.2);
  \draw[blue, very thick, ->]  (9,-1.2) -- (9.8,-1.2);
  \draw[blue, very thick, ->]  (9.85,-1.1) -- (9.85,-0.1);
  \draw[blue, very thick, ->]  (10.15,-0.1) -- (10.15,-1.1);
  \draw[blue, very thick, ->]    (10,-1.4) to [out=-90,in=-90] (7.9,-1.4);
\end{tikzpicture}
    \caption{Example of the action of the map $g$.}
    \label{fig:g}
\end{figure}
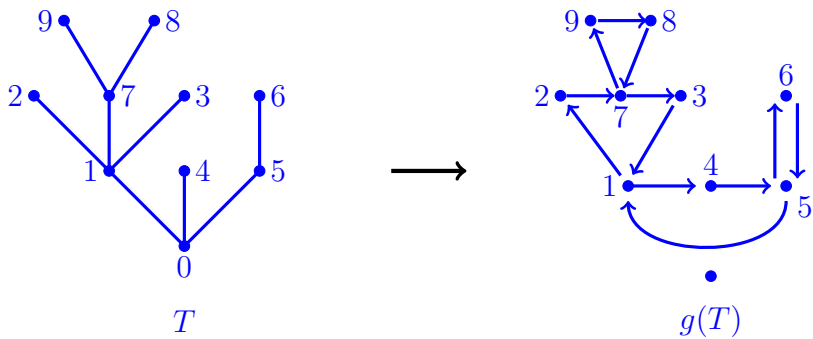

\begin{lemma}
\label{lemma:straight}
For every $T\in L'_{n}$, $g(T) \in A_n$, i.e., $g(T)$ has exactly one Eulerian tour.
\end{lemma}
\begin{proof}
First, note that every vertex $x$ of $g(T)$ belongs to one or two cycles:
\begin{itemize}
    \item The cycle where $x$ and its parent from $T$ both belong.
    \item In case $x$ has children in $T$, the cycle formed by $x$ and its children in $g(T)$.
\end{itemize}
By Lemma \ref{lemma:1}, it remains to show that there exists a unique oriented simple path between any two vertices of $g(T)$, for arbitrary $T\in L'_{n}$. To see this, observe that if $(v,w)\in E(T)$ and $v,w\neq 0$ or if both $v$ and $w$ are children of the root $0$ in $T$, then we have a unique oriented simple path between $v$ and $w$ in $g(T)$, which is part of a single cycle. For instance, the edge $(1,3)$ in the graph $T$ shown at Figure \ref{fig:g} corresponds to the oriented simple path $1273$ in $g(T)$, whereas the path $514$ in $g(T)$ corresponds to the pair of children $4$ and $5$ of the root of $T$. 

We will show that since we have a unique non-oriented path $\mathcal{P}$ between any two vertices $u$ and $v$ in $T$, where $u,v\neq 0$, we will also have a unique oriented simple path $\mathcal{P}^{\text{or}}$ between $u$ and $v$ in $g(T)$. If the vertex $0$ is part of $\mathcal{P}$, then we must have vertices $h_{1}$ and $h_{2}$ in $T$, such that the edges $(h_{1},0)$ and $(0,h_{2})$ are part of $\mathcal{P}$ (since $v,w\neq 0$). Hence, $h_{1}$ and $h_{2}$ are two children of the root $0$. Replace the edges $(h_{1},0)$ and $(0,h_{2})$ of $\mathcal{P}$ with the unique oriented simple path between $h_{1}$ and $h_{2}$ in $g(T)$. Replace all the other edges of $\mathcal{P}$ with the corresponding oriented simple paths to obtain $\mathcal{P}^{\text{or}}$. For example, the unique path $\mathcal{P}$ between $3$ and $8$ in the graph $T$ on Figure \ref{fig:g} is comprised of the edges $(3,1)$, $(1,7)$, $(7,8)$. The oriented paths corresponding to these edges are $31$, $127$ and $798$, respectively. The union of these paths, namely $312798$, gives the unique path $\mathcal{P}^{\text{or}}$ between $3$ and $8$ in $g(T)$. Another example is the path $71056$ in $T$, which transforms to the path $73156$ in $g(T)$.
\end{proof}

\begin{lemma}
\label{lemma:inverse}
For every digraph $D\in A_{n}$, there exists a unique labeled tree $T\in L'_{n}$, for which $g(T)=D$. i.e., $g$ has an inverse.
\end{lemma}
\begin{proof}
Let $D$ be an arbitrary digraph in $A_{n}$. Below, we describe the procedure $build-subtree(r,C,T,D)$ that will be used to obtain the tree $T$, for which $g(T)=D$. The first argument, $r$, is a vertex in $D$ and the second argument, $C$, is a cycle in $D$ that contains $r$.

\begin{pseudo}
$build-subtree(r,C,T,D)$: \\+
For each $v \in C$: \\+
if $v\neq r$: \\+
add an edge $(r,v)$ to $E(T)$. \\
if $v$ is part of a cycle $C_{1}\neq C$: \\+
$build-subtree(v,C_{1},T)$.
\end{pseudo}
Initially, let $T$ be an empty tree with $n+1$ vertices, i.e., let $V(T)=\{0,1,\ldots ,n\}$ and let $E(T)=\emptyset$. Take the vertex with label $1$ in $D$. By Lemma \ref{lemma:1}, this vertex belongs to one or two simple oriented cycles. Suppose that the vertex belongs to one such cycle and let $C$ denotes this cycle. Then, run $build-subtree(1,C,T,D)$ and add an edge $(0,1)$ to $E(T)$. One can easily show that the resulting graph, $T$, is a tree in $L'_{n}$. First, $T$ is connected since the execution of $build-subtree(1,C,T,D)$ will reach every vertex of $D$ and connect this vertex to an already reached vertex. In addition, if we have two different paths between two vertices $u$ and $v$ of $T$, then we will be able to find two simple oriented paths between $u$ and $v$ in $D$, which contradicts Lemma \ref{lemma:1}. Finally, $n\geq 2$ and the only vertex of $T$, which is not in the subtree with root $1$, is the vertex $0$. Thus, $2$ is in that subtree and $T\in L'_{n}$.

Now, suppose that the vertex $1$ belongs to two different cycles $C_{1}$ and $C_{2}$. Then, find the unique oriented simple path between $1$ and $2$ in $D$. This path has to have an edge in common with either $C_{1}$ or $C_{2}$, but not with both. Otherwise, we will have a contradiction with Corollary \ref{corr:lem1}. Without lose of generality, let this be $C_{2}$. Execute $build-subtree(0,C_{1},T,D)$. The graph $T$, obtained at the end, will be a tree in $L'_{n}$ (see Figure \ref{fig:inverse}).
\end{proof}

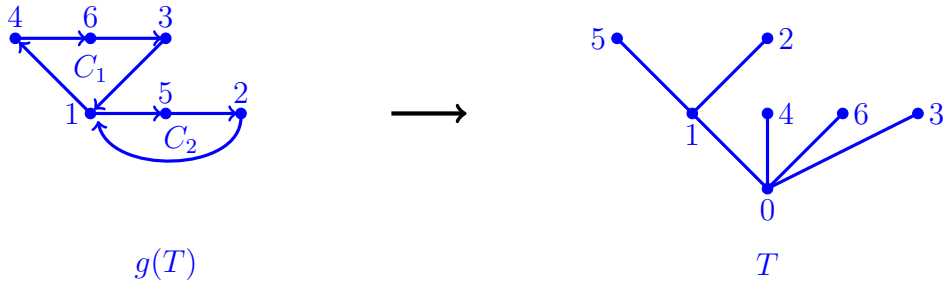
\begin{figure}[h!]
    \centering
\begin{tikzpicture}
  \draw[blue] (3,-3) node {$g(T)$};
  \filldraw[blue] (2,-1) circle (2pt) node[anchor=east] {1};
  \filldraw[blue] (3,-1) circle (2pt) node[anchor=south] {5};
  \filldraw[blue] (4,-1) circle (2pt) node[anchor=south] {2};
  \draw[blue, very thick, ->]    (2,-1) to [out=0,in=180] (2.95,-1);
  \draw[blue, very thick, ->]    (3,-1) to [out=0,in=180] (3.95,-1);
  \draw[blue, very thick, ->]    (4,-1.05)to [out=-90,in=-80] (2.1,-1.1);
  \draw[blue] (3.2,-1.35) node {$C_{2}$};
  
  \filldraw[blue] (1,0) circle (2pt) node[anchor=south] {4};
  \draw[blue, very thick, ->]    (2,-1)to [out=135,in=-45] (1.05,-0.05);
  \filldraw[blue] (2,0) circle (2pt) node[anchor=south] {6};
  \draw[blue, very thick, ->]    (1,0)to [out=0,in=180] (1.95,0);
  \filldraw[blue] (3,0) circle (2pt) node[anchor=south] {3};
  \draw[blue, very thick, ->]    (2,0)to [out=0,in=180] (2.95,0);
  \draw[blue, very thick, ->]    (3,0)to [out=-135,in=45] (2.05,-0.95);
  \draw[blue] (2,-0.4) node {$C_{1}$};
  
  \draw[->, ultra thick] (6,-1) -- (7,-1);
        \draw[blue] (11,-3) node {$T$};
  \draw[blue, very thick]  (9,0) -- (10,-1); 
  \filldraw[blue] (10,-1) circle (2pt) node[anchor=north] {1};
  \filldraw[blue] (9,0) circle (2pt) node[anchor=east] {5};
  \draw[blue, very thick]  (11,0) -- (10,-1); 
  \filldraw[blue] (11,0) circle (2pt) node[anchor=west] {2};
  \draw[blue, very thick]  (13,-1) -- (11,-2); 
  \filldraw[blue] (13,-1) circle (2pt) node[anchor=west] {3};
  \draw[blue, very thick]  (12,-1) -- (11,-2); 
  \filldraw[blue] (12,-1) circle (2pt) node[anchor=west] {6};
  \draw[blue, very thick]  (11,-1) -- (11,-2); 
  \filldraw[blue] (11,-1) circle (2pt) node[anchor=west] {4};
  \draw[blue, very thick]  (10,-1) -- (11,-2); 
  \filldraw[blue] (11,-2) circle (2pt) node[anchor=north] {$0$};
\end{tikzpicture}
    \caption{Example of the action of the inverse map $g^{-1}$.}
    \label{fig:inverse}
\end{figure}

Lemmas \ref{lemma:straight} and \ref{lemma:inverse} imply that the map $g$ is a bijection.

\section{A bijection with parentheses arrangements}

In this section, we give a second combinatorial proof of Theorem 1, via a bijection between the digraphs in $A_{n}$ and a set of valid parentheses arrangements. Suppose that you have $n$ pairs of opening and closing parenthesis, such that the two parenthesis in each pair are labeled with the numbers in $[n]$. A valid labeled parentheses arrangement is an ordering of these $2n$ parenthesis, such that we cannot have two interlaced pairs, e.g., $(_{i}(_{j})_{i})_{j}$ for some $i,j\in [n]$. A valid unlabeled parentheses arrangement is a sequence of unlabeled opening and closing parentheses that can be obtained by forgetting the labels of a valid labeled arrangement. One can easily check that a sequence of $n$ opening and $n$ closing parentheses is valid if and only if every prefix of the sequence has at least as many opening parentheses as closing parentheses. The number of valid arrangements of $n$ unlabeled pairs of parentheses is $C_n$ \cite[Theorem 1.5.1]{catalan}. Thus the number of such arrangements for labeled pairs is $n!C_n$.

To construct a bijection with the set of digraphs $A_{n}$, let us first note that one can assume that all vertices of an Eulerian digraph have indegree and outdegree $2$, provided that we allow digraphs with loops.

\begin{lemma}
If $B_n$ is the set of all Eulerian digraphs on the vertex set $[n]$ (possibly with loops) with all vertices of indegree and outdegree $2$, then $|B_n|=|A_n|$
\end{lemma}
\begin{proof}
Given a digraph $D$ in $A_n$, Corollary \ref{corr:BEST} implies that all vertices have outdegree $1$ or $2$ and the same indegree. Moreover, the single Eulerian tour of $D$ passes exactly once through each vertex of outdegree $1$. Hence, adding a loop to every vertex of outdegree $1$ gives an element of $B_n$.

Conversely, given a digraph $D'$ in $B_n$, deleting all loops gives an element of $A_n$. Indeed, the loopless digraph still has a unique Eulerian tour, which is just the tour for $D'$ without the loops (the uniqueness follows because adding back loops must give an Eulerian tour for $D'$). These two maps are inverses of each other and thus give a bijection between $A_n$ and $B_n$.
\end{proof}

Now, let $B_n^*$ be the set of digraphs in $B_n$ together with an identified edge. Since all digraphs in $B_n$ have $2n$ edges, we have $|B_n^*|=2n|B_n|=2n|A_n|$. We will give a bijection between $B_n^*$ and the set of valid arrangements of $n$ labeled pairs of parentheses, which will show that $2n|A_n|=n!C_n$, that is, $|A_n|=\frac{1}{2}(n-1)!C_n$.

\begin{theorem}
There exists a bijection between $B_n^*$ and the set of valid arrangements of $n$ labeled pairs of parentheses.
\end{theorem}

\begin{proof}

[first part: Digraphs in $B_{n}^{*}$ $\to$ valid parentheses arrangements].

Let $D$ be a digraph in $B_n^*$ with identified edge $e$. We define a parentheses arrangement $h(D)$ as follows:

Following the unique Eulerian tour of $D$, starting at $e$, open the $i$ - th pair of parentheses when you pass through the vertex $i$ for the first time and close the $i$ - th pair of parentheses when you pass through the vertex $i$ for the second time (see Figure \ref{fig:par_example} below) .
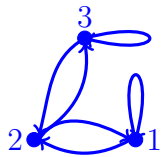
\begin{figure}[h!]
    \centering
    \begin{tikzpicture}[scale = 1.35]
   \filldraw[blue] (1,0) circle (2pt) node[anchor=east] {2};
  \draw[blue, very thick, ->]  (1.05,0) to [out=60,in=150] (1.95,0);
  \draw[blue, very thick, ->]  (1.95,0) to [out=210,in=-30] (1.05,0);
  
  \filldraw[blue] (2,0) circle (2pt) node[anchor=west] {1};
  \path[blue, very thick, ->] (2,0) edge [loop above]  (2,0);
  \filldraw[blue] (1.5,1) circle (2pt) node[anchor=south] {3};
  \path[blue, very thick, ->] (1.5,1) edge [loop right]  (1.5,1);
  \draw[blue, very thick, ->]  (1.5,1) to [out=210,in=75] (1.05,0.05);
  \draw[blue, very thick, ->]  (1.05,0.05) to [out=30,in=-80] (1.5,0.95);
  \end{tikzpicture}
    \caption{The digraph in $B_{3}^{*}$ (with the edge $2\to 1$ being identified) that yields the string $(_{1})_{1}(_{2}(_{3})_{3})_{2}$}
    \label{fig:par_example}
\end{figure}

To show that the resulting string of parentheses is valid, we have to show that we cannot have two interlaced pairs of parentheses, e.g., $$\cdots(_i\cdots (_j \cdots )_i\cdots )_j \cdots.$$ In other words, the unique tour cannot have the form $$i \xrightarrow[]{a} j \xrightarrow[]{b} i \xrightarrow[]{c} j \xrightarrow[]{d} i$$ for some walks $a,b,c,d$. But this is clearly impossible, because otherwise we would have a second Eulerian tour $i \xrightarrow[]{a} j \xrightarrow[]{d} i \xrightarrow[]{c} j \xrightarrow[]{b} i$.

[second part: Valid parentheses arrangements $\to$ Digraphs in $B_{n}^{*}$] Given a valid parentheses arrangement $w=(_x \cdots )_y$, we obtain a digraph $h^{-1}(w)\in B_n^*$ by putting an edge between the corresponding vertices of any pair of consecutive parentheses (from the first parenthesis to the second) and an edge from $y$ to $x$. The identified edge is $y\to x$.

Clearly, every vertex in $h^{-1}(w)$ has indegree and outdegree $2$ and that there exists an Eulerian tour $T$, given by the order of the parentheses' labels in $w$. Hence, we just have to show that $T$ is the unique Eulerian tour of $h^{-1}(w)$. Let $i\in[n]$ and let $$w = \cdots ?_\ell (_i (_j \cdots )_i ?_k \cdots,$$ where $?$ represent either a closing or an opening parenthesis (if $(_i$ and $)_i$ are consecutive we let $j=i$, if $(_i$ is the first parenthesis of $w$ we let $\ell$ be the label of the last one and if $)_i$ is the last parenthesis, we let $k$ be the label of the first one). We have to show that if an Eulerian tour enters the vertex $i$ for the first time from $\ell$, this tour must exit the vertex $i$ towards $j$ and not $k$. Indeed, if this is true for all $i$, then the Eulerian tour is entirely determined by its first edge. Thus, this tour and $T$ are equal up to a cyclic shift. Suppose, for the sake of contradiction, that there exists an Eulerian tour $T'$ of $h^{-1}(w)$, which exits $i$ towards $k$, after entering $i$ for the first time, through $\ell$. 

Note that, by the properties of valid parentheses arrangements, the two parentheses corresponding to any vertex $v\ne i$ are either both between $(_i$ and $)_i$ (then we will say that $v$ is of type $A$) or both outside (type $B$). 
Clearly, all edges of the graph $h^{-1}(w)$ with initial vertex of type $A$ (respectively $B$) have a final vertex either $i$ or of type $A$ (respectively $B$), so the only way to go from a vertex of type $A$ to a vertex of type $B$ is through $i$ and vice-versa. Therefore, since $k$ is of type $B$, we must eventually enter $i$ in $T'$ through a vertex of type $B$, in order to access vertices of type $A$. The only way to do so, however, is through the edge $\ell \to i$, which was already used in $T'$. This is a contradiction, so the uniqueness of the Eulerian tour is proved. The two described maps $h$ and $h^{-1}$ are obviously inverses of each other, so the proof is complete. 

\end{proof}

\section{Acknowledgements}
We are thankful to Jan Kyncel for the suggestions in the mathoverflow post related to the considered problem \cite{mathover}.

\end{document}